\documentclass[11pt,letterpaper]{amsart}
\usepackage{hyperref}
\usepackage{breqn,cleveref}
\usepackage{microtype}
\usepackage{comment}
\usepackage{setspace}
\usepackage{bm}
\usepackage{amssymb,amsmath,xcolor}
\usepackage{mathtools}
\mathtoolsset{showonlyrefs}
\usepackage{graphicx}

\newcommand{\bi}{\bar{i}}

\newcommand{\bk}{\bar{k}}

\newcommand{\bq}{\bar{q}}

\newcommand{\pbp}{\partial \bar{\partial}}

\newcommand{\fRe}{\mathfrak{Re}}

\newcommand{\ol}{\overline}
\newcommand{\ul}{\underline}

\newtheorem{theorem}{Theorem}[section]
\newtheorem{lemma}[theorem]{Lemma}

\newtheorem{corollary}[theorem]{Corollary}

 \theoremstyle{definition}

\theoremstyle{remark}

\numberwithin{equation}{section}

\begin{document}

\title[Geodesic equation in the space of mixed Volume Forms]{Regularity of a Geodesic equation in the space of mixed Volume Forms on Hermitian Manifolds}

\author{Mathew George}
\address{Department of Mathematics, Purdue University,
         West Lafayette, IN 47907, USA}
\email{georg233@purdue.edu}

\date{}

\begin{abstract}

We prove regularity of a fully nonlinear equation that arises from the study of geodesics in the space of mixed volume forms on Hermitian manifolds admitting a balanced metric. Under conditions for ellipticity, we prove that this degenerate equation has a $C^{1,1}$ solution on Hermitian manifolds. We derive uniform Laplacian estimates for the perturbed equation, and also construct explicit subsolutions. In particular, this shows the existence of a unique $C^{1,1}$ solution to the Donaldson equation on Hermitian manifolds.

\end{abstract}


\maketitle

\section{Introduction}

The study of geodesic equations in geometric spaces have led to powerful machinery in the complex geometry. The most notable application in this direction is the existence and uniqueness of canonical metrics on compact K\"ahler manifolds. In a different context, the space of volume forms on Riemannian manifolds was introduced by Donaldson \cite{Donaldson10} and studied extensively by Chen-He \cite{CH11} and He \cite{He08}. This was introduced in connection to a free-boundary problem associated to the Nahm's equation. Similar type of equations have been observed in the Gursky-Streets framework \cite{GS19,He21} of the Yamabe problem. Gursky and Streets introduced a Riemannian metric on the space of conformal metrics on a Riemannian manifold that leads to a degenerate fully nonlinear equation. Motivated by these progresses, in \cite{George24} we studied the space of mixed volume forms $\mathcal V_p$ on a balanced manifold $(M, \omega)$ and a similar type of PDE is obtained:

\begin{equation} \label{I1.1}
\begin{cases}
    &\phi_{ t t } ( n + n X \phi + \Delta \phi) - | \nabla \phi_t |^2 = - \dfrac{ n X \phi_t^2 }{2},\\
     & \phi(z,0) = \phi_0(z),\;\;\; \phi(z,1) = \phi_1(z) \in \mathcal{V}_p
    \end{cases}
\end{equation}

\

However this equation is quite different since it involves the geometric quantity $X \omega^n= \sqrt{-1} \pbp \omega^{p-1} \wedge \omega^{n-p}$, the sign of which determines the type of the equation. In this paper, we derive $C^{1,1}$ estimates for this equation on Hermitian manifolds, assuming $X\leq 0$ which makes it degenerate elliptic. In the particular case when $X=0$, this shows the existence of $C^{1,1}$ solutions to the Donaldson's equation:

\begin{equation}\label{Donaldson-eqn}
    \begin{cases}
     &\phi_{tt}(n + \Delta \phi) - |\nabla \phi_t|^2 = 0,\\
     &\phi(z,0)=\phi_0(z),\;\; \phi(z,0)=\phi_0(z)  
    \end{cases}
\end{equation}

\

\noindent on Hermitian manifolds $M \times [0,1]$ with smooth boundary data. 

\

Fix any $1 \leq p \leq n$. Given a Hermitian manifold $(M, \omega)$, let $\Omega_{\phi}=\omega^{p}+\sqrt{-1}\pbp(\phi \omega^{p-1})$. Consider the space of mixed volume forms given by 

$$ \mathcal{V}_p = \{ \phi\in C^{\infty}(M) : \Omega_{\phi} \wedge \omega^{n-p} >0 \},$$

\noindent with a Riemannian metric

$$(\psi_1, \psi_2)_{\phi} = \left(\int_M\psi_1 \psi_2 \; \Omega_{\phi} \wedge \omega^{n-p}\right)^{1/2}$$

\

\noindent for $\psi_1, \psi_2 \in T_\phi \mathcal V_p$. Assuming in addition that $d \omega^{n-1} = 0$ i.e. the metric is balanced, a geodesic equation \cite{George24} in the domain $M \times [0,1]$ is given by \eqref{I1.1}.

\

 While the case of balanced metrics yield $X \geq 0$ \cite{George24} forcing the equation \eqref{I1.1} to be degenerate hyperbolic, this paper deals with the regularity of this equation when $X \leq 0$. Under this condition, it becomes degenerate elliptic. We plan to consider the geometrically distinct hyperbolic case in future works.

\

The main result of the paper is the following.

\begin{theorem}
    Let $(M, \omega)$ be a Hermitian manifold and $\psi >0$ be a smooth function. Then assuming $X \leq 0$, any solution $\phi \in C^{4}(M \times [0,1])$ to the equation 
    
    \begin{equation} \label{I1.2}
\begin{cases}
    &\phi_{ t t } ( n + n X \phi + \Delta \phi) - | \nabla \phi_t |^2 = \psi - \dfrac{ n X \phi_t^2 }{2},\\
     & \phi(z,0) = \phi_0(z),\;\;\; \phi(z,1) = \phi_1(z) \in \mathcal{V}_p
    \end{cases}
\end{equation}

   \noindent satisfies

    $$|\Delta \phi|_{C^0}\; + \;|\nabla \phi|_{\omega} \; \leq \; C, $$

    \

    \noindent for a uniform constant $C$ that depends only on the background data.

\end{theorem}

\

This is sufficient for showing the existence of $C^{1,1}$ weak solutions to \eqref{I1.1}.

\begin{corollary}\label{corollary}
    There exists a $C^{1,1}$ weak solution to the equation 
 \begin{equation}
            \phi_{tt}(n + nX \phi + \Delta \phi) - |\nabla \phi_t|^2 =  - \dfrac{nX \phi_t^2}{2}
    \end{equation}

    \noindent satisfying the boundary conditions $\phi(z,0) = \phi_0(z), \;\;\; \phi(z,1) = \phi_1(z)$ in $\mathcal{V}_p$. If $X=0$ this solution is unique.
\end{corollary}

Here is a rough sketch of the arguments in this paper. First, we derive estimates for the Laplacian $\Delta  \phi$, that depends linearly on the gradient $\nabla \phi$. Then $|\nabla \phi|_{\omega}$ estimate is obtained by interpolation inequalities. From the previous work \cite{George2024}, we already have estimates for $|\phi|$, $|\phi_{t}|$ and $|\phi_{tt}|$ needed to state the final result. The $\phi_{tt}$ estimate is conditional on the existence of a subsolution, which we construct explicitly in Section $2$. Finally, the Evans-Krylov theorem is used to bound the full complex Hessian. None of the estimates uses the balanced condition and could be stated in generality on Hermitian manifolds.

\

\section*{Acknowledgement and AI tool use}

The author would like to thank Nicholas McCleerey for helpful discussions and Xi Sisi Shen for helpful comments.

\

Gemini was used for literature search.

\

\section{Preliminaries}

\bigskip

We fix some notations first. Let  $(M, \omega)$ be a Hermtian manifold and $X =  M \times [0,1]$ with the corresponding product metric. $\Delta$ will denote the Laplacian with respect to the Dolbeault operator $\partial$ given by

$$\Delta  = \partial \partial^* + \partial^* \partial .$$

\

For a smooth function $\phi \in C^{\infty}(M)$,  it can be written as 

$$\Delta \phi = \partial^* \partial\phi = \mbox{tr}_{\omega} (i \pbp \phi).$$

\

The real part of the Hermitian metric defines a Riemannian metric $g$ on $M$ whose Laplace-Beltrami operator is given by 

$$\Delta_g = dd^* + d^* d.$$

They are related by

\begin{equation}\label{Laplacians}
    \Delta_g \phi = 2 \Delta \phi + g(d\phi, \theta),
\end{equation}

\

\noindent for $\theta = J d^* \omega$ involves the torsion of $\omega$. The complex Hessian and real Hessians will be denoted by $i \pbp \phi$ and $D^2 \phi$ respectively.

\

\textbf{Ellipticity condition:} We show that $X \leq 0$ is necessary for degenerate ellipticity. Denote $A(\phi) = n+ nX \phi + \Delta \phi.$ Then the symbol of the differential operator is given in local orthonormal coordinates $\{z_k\}$ by

$$ \sigma(\phi) =  \begin{bmatrix}
    A(\phi) & -\phi_{k t} \\
    -\phi_{\bk t} & \phi_{tt}
\end{bmatrix}.$$

\ 

Degenerate ellipticity requires this matrix to be positive semi-definite, which would imply that its Schur complement

$$A(\phi) - \begin{bmatrix}
    -\phi_{1 t} & -\phi_{2 t}& \hdots &-\phi_{n t}
\end{bmatrix} \mbox{Diag}
     \left[\frac{1}{\phi_{tt}}, \frac{1}{\phi_{tt}},\hdots,  \frac{1}{\phi_{tt}}\right]\begin{bmatrix}
    -\phi_{{\bar 1} t} \\ -\phi_{{\bar 2} t} \\ \vdots \\ -\phi_{{\bar n} t}
\end{bmatrix} $$

\

\noindent is positive semi-definite. But this is equal to 

$$\frac{1}{\phi_{tt}} (A(\phi)\phi_{tt} - |\nabla \phi_t|^2) = \frac{-nX\phi_t^2}{2\phi_{tt}}.$$

\

So if $X>0$, the PDE cannot be degenerate elliptic.

\bigskip

\textbf{Subsolutions:} We construct a subsolution of the form 

\begin{equation}
    \ul{\phi}(z,t) = (1-t) \phi_0(z) + t \phi_{1}(z) + g(t)
\end{equation}

\

\noindent for some function $g(t)$. We reduce it to solving an ODE for the function $g$. Assume $X<c<0$. If $X=0$, the subsolution from \cite{CH11} given by

$$\ul{\phi} = (1-t)\phi_0 + t\phi_1  + at(1-t),$$

\

\noindent for some carefully chosen $a>0$ will be sufficient.

\

Calculating the derivatives of $\ul{\phi}$,

\begin{equation}
    \begin{aligned}
        &\Delta \ul{\phi} = (1-t) \Delta \phi_0 + t \Delta \phi_1, \quad \; &\nabla \ul{\phi}_t = \nabla (\phi_1 - \phi_0)\\
        & \ul{\phi}_t = g'(t) + (\phi_1 -\phi_0), &{\ul{\phi}}_{tt} = g''(t)
    \end{aligned}
\end{equation}

\

This gives

\begin{equation}\label{subsolution1}
    \begin{aligned}
        \ul{\phi}_{tt}(n + n X \ul{\phi} + \Delta \ul{\phi}) - |\nabla \ul{\phi}_t|^2 + \frac{nX \ul{\phi}_t^2}{2} =& \;g''(t)\left[tA(\phi_1)+ (1-t) A(\phi_0) + nX g(t)\right] \\
         - &|\nabla(\phi_1 - \phi_0)|^2 + \frac{nX}{2}\left[g'(t) + (\phi_1 -\phi_0)\right]^2\\
        \geq & \; g''(t)[c_0 + nXg(t)] + nX g'(t)^2-K
    \end{aligned}
\end{equation}

\

\noindent where $c_0 = \inf_{X\times [0,1]}\{tA(\phi_1)+ (1-t) A(\phi_0)\} >0 $, since each $A(\phi_0)$ and $A(\phi_1)$ are separately bounded below by uniform positive constants. Also

$$K = \sup_{X}\{|\nabla(\phi_1 - \phi_0)|^2 + n|X| (\phi_1 -\phi_0)^2\}.$$

\

It is enough to show that the RHS of \eqref{subsolution1} is strictly positive.

\ 

Define $y = c_0 + nXg(t)$. Then $g''(t) = \dfrac{y''}{nX}$ and $g'(t) = \dfrac{y'}{nX}$. So we can write

$$g''(t)[c_0 + nXg(t)] + nX g'(t)^2 -K = -\frac{1}{n|X|}\left(y''y + {y'}^2 - nKX\right).$$

This is an ODE that can be solved explicitly. We look for solutions of

\begin{equation}
    y''y + {y'}^2 = - A
\end{equation}

for $A>0$ and $y(0) = y(1) = c_0$. The solution is given by

\begin{equation}
    y^2 = -At^2 + (c_0 - c_0^2 +A)t + c_0^2.
\end{equation}

\

Choose a constant $A$ such that $A > \sup{(n K |X|)}$. For this $A$,

$$ -At^2 + (c_0 - c_0^2 +A)t + c_0^2 = c_0t + c_0^2(1-t) + At(1-t) >0$$

\noindent is strictly positive. So $y$ is well-defined and $y''y + {y'}^2 - nKX$ is negative as required. Finally we have an explicit subsolution:

$$\ul{\phi}(z,t) =(1-t) \phi_0(z) + t \phi_{1}(z) + \frac{\sqrt{-At^2 + (c_0 - c_0^2 +A)t + c_0^2} -c_0}{nX},$$

\

\noindent for $A$, $c_0$, and $K$ as defined above.

\section{Laplacian estimates on Hermitian manifolds}

\

We first show that the Laplacian estimates can be obtained by a shorter argument if $X=0$. That is, if the equation is given by 

$$\phi_{tt}(n + \Delta \phi) - |\nabla \phi_t|^2 = \psi(z),$$

\

\noindent for $\psi>0$. For $\psi=0$, this matches the classical version of the equation \eqref{Donaldson-eqn} studied in \cite{CH11} and \cite{He08} on Riemannian manifolds. The torsion terms in the Hermitian case causes some difficulties that cannot be addressed by previous methods. We estimate the Laplacian by considering the function

$$Q = \log(n +\Delta \phi) + h(t)$$

\

\noindent for function $h$ to be fixed later. Denote the linearization of the PDE at $\phi$ by $\mathcal L_{\phi} v$

$$\mathcal L_{\phi} v = \phi_{tt} \Delta v + (n+\Delta \phi) \;v_{tt} - 2 \fRe{(\phi_{tk} v_{\bk t})}.$$

\

Differentiate the PDE by $\Delta$ to get

\begin{equation}\label{diff-pde}
    \begin{aligned}
        \Delta \psi &= \Delta \phi_{tt} (n+ \Delta \phi) + \phi_{tt} \Delta \Delta \phi - 2 \fRe{(\phi_{t k} \Delta \phi_{t \bk})}\\
        & \;\;\;+\sum\limits_i \left(2\fRe(\phi_{tt i} \Delta \phi_{\bi}) - |\nabla_i \phi_{t k} |^2- |\nabla_{\bi} \phi_{t k}|^2\right)\\
        & = \mathcal L_{\phi} (\Delta \phi) + \sum_k(\phi_{tk}\ol{T^q_{i k}} \phi_{i \bq t} + \phi_{t \bk}T^q_{i k} \phi_{q \bi t} - \phi_{t \bk} R_{i \bi k}^q \phi_{q t}) \\
        &\;\;\;+\sum\limits_i \left(2\fRe(\phi_{tt i} \Delta \phi_{\bi}) - \sum_{k}( |\nabla_i \phi_{t k} |^2 + |\nabla_{\bi} \phi_{t k}|^2)\right)\\
    \end{aligned}
\end{equation}

\noindent where we used the commutation formulae

\begin{equation}\label{commute1}
    \Delta \nabla_{\bk}\phi - \nabla_{\bk} \Delta \phi = - \ol{T^q_{i k}} \phi_{i \bq}.
\end{equation}

\ 

\begin{equation}\label{commute2}
    \Delta \nabla_{k}\phi - \nabla_{k} \Delta \phi = - T^q_{i k} \phi_{q \bi} + R_{i \bi k}^q \phi_{q} .
\end{equation}

Now since 

$$\sum_k (\phi_{tk}\ol{T^q_{i k}} \phi_{i \bq t} + \phi_{t \bk}T^q_{i k} \phi_{q \bi t}) \leq C |\nabla \phi_t|^2 + \epsilon \sum_{i,k} |\nabla_{\bi} \phi_{k t }|^2,$$

\noindent \eqref{diff-pde} implies

\begin{equation}
    \begin{aligned}
        \mathcal{L}_{\phi} (\Delta \phi) \geq \Delta \psi -C| \nabla \phi_t|^2 + (1-2\epsilon)\sum_{i,k} ( |\nabla_i \phi_{t k} |^2 + |\nabla_{\bi} \phi_{t k}|^2)
    \end{aligned}
\end{equation}

\noindent where we used that at the maximum point $\nabla_{\bi} Q = 0$, and hence $\Delta \phi_{\bi} =0$.

\begin{equation}\label{max-principle}
    \begin{aligned}
        \mathcal{L}_{\phi} Q &= \frac{\mathcal{L}_{\phi} \Delta \phi}{n+\Delta \phi} - (n + \Delta \phi)\frac{|\Delta \phi_t|^2}{(n+\Delta \phi)^2} + (n+ \Delta \phi) g''(t)\\
        &\geq \frac{\Delta \psi}{n + \Delta \phi} -C\frac{|\nabla \phi_t|^2}{n+\Delta \phi} + \frac{(1-2\epsilon)}{n + \Delta \phi}\sum_{i,k} ( |\nabla_i \phi_{t k} |^2 + |\nabla_{\bi} \phi_{t k}|^2) \\
        &\;\;\;+ (n+ \Delta \phi) (g''(t) - g'(t)^2)
    \end{aligned}
\end{equation}

\

\noindent where we used that at an interior maximum point $\partial_t \Delta \phi = -g'(t)(n + \Delta \phi)$. Note that from the PDE we have the following inequality

$$|\nabla \phi_t|^2  \leq (n+ \Delta \phi)\phi_{tt},$$

\noindent so that $\dfrac{|\nabla \phi_t|^2}{n + \Delta \phi} \leq C$, since $\phi_{tt}$ is bounded. Combining this with \eqref{max-principle}, it is enough to show that

$$(n+ \Delta \phi)(g''(t) - g'(t)^2) >  \frac{\Delta \psi}{n+\Delta \phi} -C .$$

\

This is satisfied by setting $g(t) = -a \log(t+b)$ where $b$ is chosen so that $1 \leq t+b \leq 2$ (assuming $t$ is in a bounded domain). Then

$$g''(t) = \frac{(g'(t))^2}{a}, \text{    and    } (g'(t))^2 \geq \frac{a^2}{4}.$$

So the above inequality reads

$$(n+ \Delta \phi)\frac{a(1-a)}{4} >  \frac{\Delta \psi}{n + \Delta \phi} -C .$$

Now choose $a = \dfrac{1}{2}$. This gives Laplacian estimates when $X=0$.

\


Next we derive Laplacian estimate for the general case when $X$ does not vanish identically. The equation is given by

\begin{equation}\label{balanced-pde}
    \phi_{tt}(n + nX \phi + \Delta \phi) - \sum |\phi_{k t}|^2 = \psi -\frac{nX \phi_t^2}{2}.
\end{equation}

\

\

Similar to before, apply $\Delta$ to \eqref{balanced-pde} to get

\begin{equation}
    \begin{aligned}
        A(\phi) \Delta \phi_{tt} &+ \phi_{tt}(n \Delta X \phi + n \sum_k 2 \fRe( \nabla_k X .\phi_{\bk}) + \Delta \Delta \phi + nX \Delta \phi)\\
        & - 2\fRe{\phi_{\bk t} \Delta \phi_{kt}}  + 2 \fRe(\phi_{tt k}(n \nabla_{\bk} X\phi+nX \phi_{\bk}+ \nabla_{\bk} \Delta \phi))\\
        &-  \sum_{i,k}(|\phi_{ikt}|^2 + |\phi_{i \bk t}|^2 ) = \Delta \psi-\frac{n}{2}\Delta(X \phi_t^2)
    \end{aligned}
\end{equation}

\noindent  which after using the commutation formulae \cref{commute1,commute2} for $\Delta \phi_{kt}$, $\Delta \phi_{\bk t}$ gives

\begin{equation}
    \begin{aligned}
     &\mathcal L_{\phi} (\Delta \phi) + \sum_k(\phi_{tk}\ol{T^q_{i k}} \phi_{i \bq t} + \phi_{t \bk}T^q_{i k} \phi_{q \bi t} - \phi_{t \bk} R_{i \bi k}^q \phi_{q t}) \\
        &+ \phi_{tt}(n \Delta X \phi + n \sum_k 2 \fRe( \nabla_k X .\phi_{\bk}) + nX \Delta \phi) \\
        &+ 2 \fRe(\phi_{tt k}(n \nabla_{\bk} X\phi+ \nabla_{\bk} \Delta \phi + nX \phi_{\bk})) -  \sum_{i,k}(|\phi_{ikt}|^2 + |\phi_{i \bk t}|^2 )\\
        &= \Delta \psi -\frac{n}{2}\Delta(X \phi_t^2).
    \end{aligned}
\end{equation}

\

Then as before absorbing the torsion and curvature terms, we get

\begin{equation}\label{diff-eq}
    \begin{aligned}
        \mathcal L_{\phi} (\Delta \phi) \geq & - C |\phi_{tk}|^2
    - \phi_{tt}(n \Delta X \phi + n \sum_k 2 \fRe( \nabla_k X .\phi_{\bk})  + nX \Delta \phi)  \\
        &- 2 \fRe(\phi_{tt k} (n\nabla_{\bk} X.\phi+ \nabla_{\bk} \Delta \phi + nX \phi_{\bk})) +  (1-\epsilon) \sum_{i,k}(|\phi_{ikt}|^2 \\
        &+ |\phi_{i \bk t}|^2 )+ \Delta \psi -\frac{n}{2}\Delta(X \phi_t^2).
    \end{aligned}
\end{equation}

\

Here the problem term is $- 2n \fRe(\phi_{tt k} \nabla_{\bk} X.\phi)$. We modify the test function accordingly.


\

Let $L = \sup (n + n X \phi)$. Consider $Q(u) = \log(\Delta \phi + L) + g(t) + b \phi_{tt},$ for a small constant $b>0$. Then at the point of maximum

\begin{equation} \label{eq1}
    \nabla_{\bk} \Delta \phi = - b \phi_{t t \bk} ( L + \Delta \phi),
\end{equation}

\noindent so that 

\begin{equation}\label{eq2}
    - 2n \fRe(\phi_{tt k} \nabla_{\bk} \Delta \phi) = 2nb |\phi_{tt k}|^2 ( L +\Delta \phi).
\end{equation}

\noindent and

\begin{equation}\label{eq3}
   -\phi_{tt} \sum_k\frac{|\nabla_k \Delta \phi|^2}{ (L + \Delta \phi)^2} = -\phi_{tt} b^2 \sum_k|\phi_{tt k}|^2. 
\end{equation}

\begin{lemma}

\ 

    \begin{equation}
        \mathcal{L}_{\phi} \left( \log(L + \Delta \phi) + g(t) + b \phi_{tt}\right) \geq - C \frac{|\nabla \phi|}{L + \Delta \phi}+ \frac{\delta(1-\epsilon)}{4} (L + \Delta \phi) -C
    \end{equation}

    \

    \noindent for a constant $C = C(\phi,\phi_t,\phi_{tt}, \epsilon,b)$. 
\end{lemma}

\

\begin{proof}

Applying $\mathcal L_{\phi} = \phi_{tt} \Delta + A(\phi) \partial_t\partial_t - \phi_{t \bk} \nabla_{k} \partial_t - \phi_{t k} \nabla_{\bk} \partial_t$ to $Q$ gives

\begin{equation}\label{c4}
\begin{aligned}
    \mathcal{L}_{\phi} \left( \log( L + \Delta \phi) + g(t) + b \phi_{tt}\right) =& \; \frac{\mathcal L_{\phi} (\Delta \phi)}{L + \Delta \phi} -\phi_{tt}\frac{| \nabla_k\Delta \phi|^2}{(L + \Delta \phi)^2} + A(\phi) g''(t) + b \mathcal L_{\phi}(\phi_{tt})\\
    & - \frac{1}{(L + \Delta \phi)^2}\left[A(\phi)|\partial_t \Delta \phi|^2 - 2 \fRe(\phi_{tk} \nabla_{\bk} \Delta\phi \partial_t \Delta \phi) \right].
\end{aligned}
\end{equation}

\

By \eqref{diff-eq} and \eqref{eq1}, and since $ L + \Delta \phi \geq \dfrac{|\nabla \phi_{t}|^2}{\phi_{tt}}$

\

\begin{equation}
    \begin{aligned}
       \frac{\mathcal{L}_{\phi} \Delta\phi}{ L + \Delta \phi} \geq &-C( 1+ \phi_{tt}) - C \frac{|\nabla \phi|}{L + \Delta \phi} +  2b|\phi_{ttk}|^2 - \frac{2n}{ L + \Delta \phi}\fRe(\phi_{ttk}\nabla_{\bk} X.\phi)\\
       & + \frac{ (1-\epsilon)}{L + \Delta \phi} \sum_{i,k}(|\phi_{ikt}|^2 + |\phi_{i \bk t}|^2 ) -\frac{n}{2( L +\Delta \phi)}\Delta(X \phi_t^2)- C |\phi_{ttk}|.
    \end{aligned}
\end{equation}

\

Since $\dfrac{|\Delta \phi_t^2|}{L + \Delta \phi} = \dfrac{|2\phi_t \Delta \phi_t + 2 |\nabla \phi_t|^2|}{L + \Delta \phi} \leq C\left(1+ \dfrac{|\Delta \phi_t|}{L + \Delta \phi}\right)$, we have

\

\begin{equation}\label{c1}
    \begin{aligned}
          \frac{\mathcal{L}_{\phi} \Delta \phi}{L + \Delta \phi} \geq & \left(b- \frac{1}{L + \Delta \phi}\right)|\phi_{tt k}|^2 + \frac{ (1-\epsilon)}{ L + \Delta \phi} \sum_{i,k}(|\phi_{ikt}|^2 + |\phi_{i \bk t}|^2 ) - C \dfrac{|\Delta \phi_t|}{L + \Delta \phi}\\
          &- C \left(1 + \frac{|\nabla \phi|}{L + \Delta \phi}\right).
    \end{aligned}
\end{equation}

\

The third term above is controlled by the second term as follows

$$C|\Delta \phi_{t}| \leq \frac{nC^2}{4\epsilon} + \epsilon \sum_{k}|\phi_{k \bk t}|^2.$$

Differentiating the PDE \eqref{balanced-pde} by $\partial_t \partial_t$ and using $\phi_{ttt} = -\dfrac{1}{b}\left(g'(t) + \dfrac{\Delta \phi_t}{L + \Delta \phi}\right)$ gives

\begin{equation}\label{c2}
    \begin{aligned}
      \mathcal L_{\phi} (\phi_{tt}) &= -\phi_{ttt}(nX \phi_t + \Delta \phi_t) + 2\sum_k|\phi_{ktt}|^2 -n X(\phi_t \phi_{ttt} + \phi_{tt}^2) \\
      & = \frac{|\Delta \phi_t|^2}{b( L + \Delta \phi)} + \frac{g'(t)\Delta \phi_t}{b} -\frac{2}{b}n X \phi_t \left(\frac{\Delta \phi_t}{L + \Delta \phi} + g'(t)\right) \\
      &\;\;\;-nX \phi_{tt}^2  + 2\sum_k|\phi_{ktt}|^2\\
      &\geq  \frac{1}{b}\left[\left(1- \epsilon - \delta\right)\frac{|\Delta \phi_t|^2}{ L + \Delta \phi}  - \frac{( L +\Delta \phi) g'(t)^2}{4 \delta}+ 2b \sum_k|\phi_{ktt}|^2 \;- \; C\right]
    \end{aligned}
\end{equation}
\noindent where in the last line we used 

$$|g'(t)\Delta \phi_t| \leq \frac{(L +\Delta \phi) g'(t)^2}{4 \delta} + \frac{\delta}{L + \Delta \phi} |\Delta \phi_t|^2,$$

\noindent and 

$$|2nX \phi_t \Delta \phi_t| \leq C + \epsilon|\Delta \phi_t|^2.$$

Finally we have

\begin{equation}\label{c3}
    \mathcal L_{\phi} (g(t)) =  A(\phi) g''(t).
\end{equation}

\

From \cref{eq1,eq3}, and $\dfrac{A(\phi)}{L + \Delta \phi} \leq (1+ \epsilon)$, it follows

\begin{equation}\label{c5}
\begin{aligned}
       - \frac{1}{(L + \Delta \phi)^2}&\left(A(\phi)|\partial_t \Delta \phi|^2  - 2 \fRe(\phi_{tk} \partial_{\bk} \Delta\phi \partial_t \Delta \phi) \right)-\phi_{tt}\sum_k\frac{| \nabla_k\Delta \phi|^2}{(L + \Delta \phi)^2}  \\
       &\geq-\frac{2b}{L + \Delta \phi} \sum_k\fRe(\phi_{tk} \phi_{t t \bk} \Delta \phi_{t})-(1+ \epsilon)\frac{|\Delta \phi_t|^2}{L + \Delta \phi}  -\phi_{tt} b^2 \sum_k|\phi_{tt k}|^2\\
       & \geq -b\phi_{tt}\frac{|\Delta \phi_t|^2}{L + \Delta \phi} - \frac{b}{\phi_{tt}( L +\Delta \phi)} |\phi_{t k}|^2 |\phi_{tt \bk}|^2 -(1+ \epsilon)\frac{|\Delta \phi_t|^2}{L + \Delta \phi}\\
       &\;\;\;-\phi_{tt} b^2 \sum_k|\phi_{tt k}|^2.
\end{aligned}
\end{equation}

In the last line, the second term could controlled by the $2b|\phi_{tt k}|^2$ from \eqref{c2}, since $L + \Delta \phi > \dfrac{|\phi_{t k}|^2}{\phi_{tt}}$.

\

Combining \cref{c1,c2,c3,c4,c5}, 

\begin{equation}
    \begin{aligned}
         \mathcal{L}_{\phi} \left( \log(L + \Delta \phi) + g(t) + b \phi_{tt}\right) \geq &- C \left(1 + \frac{|\nabla \phi|}{L + \Delta \phi}\right) + \frac{ (1-2\epsilon)}{L + \Delta \phi} \sum_{i,k}(|\phi_{ikt}|^2 + |\phi_{i \bk t}|^2 )\\ &+ \left(1- \epsilon -\delta -b\phi_{tt}\right)\frac{|\Delta \phi_t|^2}{L + \Delta \phi} -(1+ \epsilon)\frac{|\Delta \phi_t|^2}{L + \Delta \phi}\\&-\frac{(L +\Delta \phi )(g'(t))^2}{4 \delta}+ (2b- \frac{1}{L + \Delta \phi} \\
         &-\phi_{tt}b^2) \sum_k|\phi_{ktt}|^2- \; C + A(\phi)g''(t) ,
    \end{aligned}
\end{equation}

\

\noindent where $C=C(\phi,\phi_t,\phi_{tt}, \epsilon,b).$ Set $g(t) = -2 \delta \log(1+t)$, so that $g''(t) = \dfrac{g'(t)^2}{2 \delta}$ and 

$$A(\phi)g''(t)  - \frac{( L + \Delta \phi) (g'(t))^2}{4 \delta}\geq \frac{\delta(1-\epsilon)}{4}( L +\Delta \phi) .$$

\

Set $b$ small such that $b \phi_{tt} \leq \epsilon < 1$ and choose $\delta < \epsilon$. So we have

$$ \left(1- \epsilon -\delta -b\phi_{tt}\right)\frac{|\Delta \phi_t|^2}{L + \Delta \phi} -(1+ \epsilon)\frac{|\Delta \phi_t|^2}{L + \Delta \phi} \geq -4 \epsilon \frac{|\Delta \phi_t|^2}{L + \Delta \phi}.$$

Now we have  

$$4 \epsilon|\Delta \phi_t|^2 \leq 4n\epsilon \sum_k|\phi_{k \bk t}|^2 \leq (1-2\epsilon) \sum_{k} |\phi_{k \bk t}|^2.$$

Combining the above gives

\begin{equation}
\begin{aligned}
    \mathcal{L}_{\phi} \left( \log(L + \Delta \phi) + g(t) + b \phi_{tt}\right) \geq &- C \frac{|\nabla \phi|}{ L + \Delta \phi}+ \frac{\delta(1-\epsilon)}{4} (L + \Delta \phi) -C.
         \end{aligned}
\end{equation}

\end{proof}

\ 

From the maximum principle, since $\mathcal L_{\phi} (Q) \leq 0$, it follows that $\sup |\Delta \phi| \; \leq C (1 + \sup |\nabla \phi|)$.

\

\section{Gradient and higher order estimates}

\

The gradient estimates can be derived by classical interpolation inequalities. But since the second-order estimate derived is for the Laplacian, we first use the Calderon-Zygmund theorem to bound the $L^{p}$-norm of the Hessian in terms of the $L^{p}$-norm of the Laplacian. 

In an open bounded domain $\Omega \subset \mathbb{R}^{ 2n}$ for any $1< p < \infty$, \cite[Theorem~$9.9$]{GT83} gives the following estimate for the real Hessian: 

\begin{equation}\label{CZ}
    \begin{aligned}
        |D^2 \phi|_{L^{p}(\Omega)} \leq C |\Delta \phi|_{L^{p}(\Omega)}.
    \end{aligned}
\end{equation}

\

 We first translate this to the setting of Hermitian manifolds and with the Laplacian associated to $\partial$. Clearly \eqref{CZ} holds on compact Riemannian manifolds. Extension to $\partial$-Laplacian follows from \eqref{Laplacians}.

\begin{equation}
    \begin{aligned}
         |D^2 \phi|_{L^{p}(M)} \leq C (|\Delta \phi|_{L^{p}(M)} + |d\phi|_{L^p(M)})
    \end{aligned}
\end{equation}

\

Now by Gagliardo-Nirenberg interpolation inequalities \cite{Nirenberg59, Nirenberg66}, we have

\begin{equation}
|\nabla\phi|_{L^\infty} \leq C |D^2\phi|_{L^{2n}}^{2/3} |\phi|_{L^\infty}^{1/3} + C |\phi|_{L^\infty}.
\end{equation}

\

Combining the above two inequalities with $p=2n$, and the Laplacian estimates, we get

$$|\nabla\phi|_{L^\infty} \leq C (|\nabla \phi|_{L^{\infty}}^{2/3} + 1). $$

\

This gives the required gradient estimates. 

\

The boundary estimates up to second order was shown in \cite{George24}. Now global $C^{2, \alpha}$ estimates follows from the Evans-Krylov theorem, which applies since the uniform bounds found so far implies that the operator is uniformly elliptic. Higher-order estimates can be derived by bootstrapping techniques. We omit the technical details here. This shows the existence of smooth solutions for the perturbed equation. Now letting $\psi \to 0$ gives $C^{1,1}$ solutions for the degenerate equation, completing the proof of Corollary \ref{corollary}. 

\ 

\

\

{\bf Data Availability Statement:} DAS is not applicable to this article as no datasets were generated or analyzed during the current study.

\

{\bf Conflict of Interest:}
There is no conflict of interest.

\bigskip
\vspace{3cm}

\clearpage

\bibliographystyle{plain}
\bibliography{references}

@article {CH11,
    AUTHOR = {Chen, Xiuxiong and He, Weiyong},
     TITLE = {The space of volume forms},
   JOURNAL = {Int. Math. Res. Not. IMRN},
  FJOURNAL = {International Mathematics Research Notices. IMRN},
      YEAR = {2011},
    NUMBER = {5},
     PAGES = {967--1009},
      ISSN = {1073-7928,1687-0247},
   MRCLASS = {58E10 (53C23 58B20 58D17)},
  MRNUMBER = {2775873},
MRREVIEWER = {Gabjin\ Yun},
       DOI = {10.1093/imrn/rnq099},
       URL = {https://doi.org/10.1093/imrn/rnq099},
}

@incollection {Donaldson10,
    AUTHOR = {Donaldson, Simon K.},
     TITLE = {Nahm's equations and free-boundary problems},
 BOOKTITLE = {The many facets of geometry},
     PAGES = {71--91},
 PUBLISHER = {Oxford Univ. Press, Oxford},
      YEAR = {2010},
      ISBN = {978-0-19-953492-0},
   MRCLASS = {58E15 (35R35 70S15)},
  MRNUMBER = {2681687},
MRREVIEWER = {Derek\ G.\ Harland},
       DOI = {10.1093/acprof:oso/9780199534920.003.0005},
       URL = {https://doi.org/10.1093/acprof:oso/9780199534920.003.0005},
}

@article {George24,
AUTHOR = {George, Mathew},
TITLE = {Volume forms on balanced manifolds and a geodesic equation},
JOURNAL = {arXiv:2406.00995v3},
YEAR = {2024}
}

@article {George2024,
    AUTHOR = {George, Mathew},
     TITLE = {A generalized Gauduchon conjecture for $(p,p)$ forms on Hermitian manifolds},
   JOURNAL = {In preperation}
}

@book {GT83,
    AUTHOR = {Gilbarg, David and Trudinger, Neil S.},
     TITLE = {Elliptic partial differential equations of second order},
    SERIES = {Grundlehren der mathematischen Wissenschaften [Fundamental
              Principles of Mathematical Sciences]},
    VOLUME = {224},
   EDITION = {Second},
 PUBLISHER = {Springer-Verlag, Berlin},
      YEAR = {1983},
     PAGES = {xiii+513},
      ISBN = {3-540-13025-X},
   MRCLASS = {35Jxx (35-01)},
  MRNUMBER = {737190},
MRREVIEWER = {O.\ John},
       DOI = {10.1007/978-3-642-61798-0},
       URL = {https://doi.org/10.1007/978-3-642-61798-0},
}

@article {GS19,
    AUTHOR = {Gursky, Matthew J. and Streets, Jeffrey},
     TITLE = {A formal {R}iemannian structure on conformal classes and the
              inverse {G}auss curvature flow},
   JOURNAL = {Geom. Flows},
  FJOURNAL = {Geometric Flows},
    VOLUME = {4},
      YEAR = {2019},
    NUMBER = {1},
     PAGES = {30--50},
      ISSN = {2353-3382},
   MRCLASS = {53E30 (58E11)},
  MRNUMBER = {4061505},
MRREVIEWER = {Julian\ Scheuer},
       DOI = {10.1515/geofl-2019-0003},
       URL = {https://doi.org/10.1515/geofl-2019-0003},
}

@article {He21,
    AUTHOR = {He, Weiyong},
     TITLE = {The {G}ursky-{S}treets equations},
   JOURNAL = {Math. Ann.},
  FJOURNAL = {Mathematische Annalen},
    VOLUME = {381},
      YEAR = {2021},
    NUMBER = {3-4},
     PAGES = {1085--1135},
      ISSN = {0025-5831,1432-1807},
   MRCLASS = {53C18},
  MRNUMBER = {4333410},
MRREVIEWER = {Yuxin\ Ge},
       DOI = {10.1007/s00208-020-02021-5},
       URL = {https://doi.org/10.1007/s00208-020-02021-5},
}

@misc{He08,
      title={The Donaldson equation}, 
      author={Weiyong He},
      year={2008},
      eprint={0810.4123},
      archivePrefix={arXiv},
      primaryClass={math.AP},
      url={https://arxiv.org/abs/0810.4123}, 
}

@article {Nirenberg59,
    AUTHOR = {Nirenberg, L.},
     TITLE = {On elliptic partial differential equations},
   JOURNAL = {Ann. Scuola Norm. Sup. Pisa Cl. Sci. (3)},
  FJOURNAL = {Annali della Scuola Normale Superiore di Pisa. Classe di
              Scienze. Serie III},
    VOLUME = {13},
      YEAR = {1959},
     PAGES = {115--162},
      ISSN = {0391-173X},
   MRCLASS = {35.00},
  MRNUMBER = {109940},
MRREVIEWER = {L.\ Garding},
}

@article {Nirenberg66,
    AUTHOR = {Nirenberg, L.},
     TITLE = {An extended interpolation inequality},
   JOURNAL = {Ann. Scuola Norm. Sup. Pisa Cl. Sci. (3)},
  FJOURNAL = {Annali della Scuola Normale Superiore di Pisa. Classe di
              Scienze. Serie III},
    VOLUME = {20},
      YEAR = {1966},
     PAGES = {733--737},
      ISSN = {0391-173X},
   MRCLASS = {46.38},
  MRNUMBER = {208360},
MRREVIEWER = {Richard\ Beals},
}

\end{document}